\documentclass[a4paper, 11pt] {article}
\usepackage{amsfonts, amsmath, amssymb, url,dsfont,amsthm}
\usepackage{mathtools}
\usepackage{graphicx}
\usepackage[english] {babel}
\usepackage{enumitem}
\usepackage{array}
\usepackage{cellspace}
\usepackage{tikz}

\newtheorem{thm}{Theorem}[section]
\newtheorem{corollary}[thm]{Corollary}
\newtheorem{lemma}[thm]{Lemma}

\newtheorem{proposition}[thm]{Proposition}

\theoremstyle{definition}
\newtheorem{remark}[thm]{Remark}

\DeclareMathOperator{\reach}{Reach}

\newcommand{\R}{\mathbb{R}}

\newcommand{\N}{\mathbb{N}}

\newcommand{\Z}{\mathbb{Z}}
\newcommand{\La}{\mathbb{L}}

\newcommand{\eps}{\varepsilon}

\DeclareMathOperator{\indre}{int}

\DeclareMathOperator{\diam}{diam}

\newcommand{\Ha}{\mathcal{H}}

\title{Voronoi-based estimation of Minkowski tensors from finite point samples}
\author{Daniel Hug, Markus Kiderlen, and Anne Marie Svane}

\begin{document}

\maketitle

\begin{abstract} 
Intrinsic volumes and Minkowski tensors have been used to describe the geometry of real world objects. This paper presents an estimator that allows to approximate these quantities from digital images. It is based on a generalized Steiner formula for Minkowski tensors of sets of positive reach. When the resolution goes to infinity, the estimator converges to the true value if the underlying object is a set of positive reach. The underlying algorithm is based on a simple expression in terms of the cells of a Voronoi decomposition associated with the image.



\end{abstract}
\section{Introduction}
Intrinsic volumes, such as volume, surface area, and Euler characteristic, are widely-used tools to capture geometric features of an object; see, for instance, \cite{meckeEtAl,OM,milesSerra}.
Minkowski tensors are tensor valued generalizations of the intrinsic volumes, associating with every sufficiently regular compact set in $\R^d$ a symmetric tensor, rather than a scalar. They carry information about geometric features of the set such as position, orientation, and eccentricity. For instance, the volume tensor -- defined formally in Section \ref{minkowski} --  of rank $0$ is just the volume of the set, while the volume tensors of rank $1$ and $2$ are closely related to the center of gravity and the tensor of inertia, respectively. For this reason, Minkowski tensors are used as shape descriptors in materials science \cite{mickel,aste}, physics \cite{kapfer}, and biology \cite{beisbart,ziegel}. 

The main purpose of this paper is to present estimators that approximate all the Min\-kow\-ski tensors of a set $K$ when only weak information on $K$ is available. More precisely, we assume that a finite set $K_0$ which is close to $K$ in the Hausdorff metric  is known. The estimators are based on the Voronoi decomposition of $\R^d$ associated with the finite set $K_0$, following an idea of M\'{e}rigot et al.\ \cite{merigot}. 
What makes these estimators so interesting is that they are consistent; that is, they converge to the respective Minkowski tensors of $K$ when applied to a sequence  of 
finite approximations  converging to $K$ in the Hausdorff metric. 
We emphasize that the notion of `estimator' is used here in the sense of digital geometry \cite{digital}  meaning `approximation of the true value based on discrete input' and should not be confused with the statistical concept related to the inference from data with random noise. 
The main application we have in mind is the case where $K_0$ is a digitization of $K$. This is detailed in the following. 
	
As data is often only available in digital form, there is a need for estimators that allow us to approximate the {Minkowski} tensors from digital images. In a black-and-white image of a compact geometric object $K\subseteq \R^d$, each pixel (or voxel) is colored black if {its} midpoint belongs to $K$ and white otherwise. Thus, the information about $K$ contained in the image is the set of black pixel (voxel) midpoints $K_0=K\cap a\La$, where $\La$ is the lattice formed by {all} pixel (voxel) midpoints and $a^{-1}$ is the resolution. 
A natural criterion for {the reliability of} a digital estimator is that it yields the correct tensor when $a\to 0_+$. If this property holds for all objects in a given family of sets, for instance, for all sets with smooth boundary, then the estimator is called \emph{multigrid convergent} for this class.

Digital estimators for the scalar Minkowski tensors, that is, for the intrinsic volumes, are widespread in the digital geometry literature; see, e.g.,~\cite{digital,OM,OS} and the references therein. For Minkowski tensors up to rank two, estimators based on binary images are given 
in \cite{turk} for the two-dimensional and in \cite{mecke} for the three-dimensional case. Even for the class of convex sets, multigrid convergence has not been proven for any of the above mentioned estimators. The only exception are volume related quantities. 
Most of the above mentioned estimators are \emph{$n$-local} for some given fixed $n\in \N$. We call an estimator $n$-local if it depends on the image only through the histogram of all   $n\times \dotsm \times n$ configurations of black and white points. For instance, a natural surface area estimator \cite{lindblad} in three-dimensional space scans the image with a voxel cube of size   $2\times 2\times2$ and assigns a surface contribution to each observed configuration. The sum of all contributions is then the surface area estimator, which is clearly $2$-local.  The advantage of $n$-local estimators is that they are intuitive, easy to implement, and the computation time is linear in the number of pixels or voxels. 

However, many  {$n$-local} estimators are not multigrid convergent for convex sets; see \cite{am3} and the detailed discussion in Section \ref{known}.  This implies that many established estimators, like the mentioned one in \cite{lindblad} cannot be multigrid convergent for convex sets. 
All the estimators of 2D-Minkowski tensors in  \cite{turk}  are $2$-local. By the results in \cite{am3}, the estimators  for the perimeter and the Euler characteristic can thus not be multigrid convergent for convex sets. The multigrid convergence of the other estimators has not been investigated. 
The algorithms for 3D-Minkowski tensors in \cite{mecke} have as input a triangulation of the object's boundary, and the way this triangulation is obtained determines whether the resulting estimators are $n$-local or not. There are no known  results on multigrid convergence for these estimators either. Summarizing, to the best of our knowledge, this paper presents for the first time estimators of all Minkowski tensors of arbitrary rank that come with a multigrid convergence proof for a class of sets that is considerably larger than the class of convex sets.

The present work is inspired by \cite{merigot}, and we therefore start by recalling some basic notions from this paper.
For  a nonempty compact set $K$, the authors of \cite{merigot} define a tensor valued measure, which they call the \emph{Voronoi covariance measure},
defined on a Borel set $A\subseteq \R^d$ by 
\begin{equation*}
\mathcal{V}_R(K;A)  =  \int_{ K^R }\mathds{1}_A(p_K(y)) (y-p_K(y))(y-p_K(y))^\top\,dy.
\end{equation*} 
Here, $K^R$ is the set of points at distance at most $R>0$ from $K$ and $p_K$ 
 is the \emph{metric projection} on $K$: the point $p_K(x)$ is the point in $K$ closest to $x$, provided that this closest point is unique. The metric projection of $K$ is well-defined on
$\R^d$ with the possible exception of a set of Lebesgue-measure zero; see, e.g., \cite{fremlin}. 

The paper \cite{merigot} uses the Voronoi covariance measure to determine local features of surfaces. It is proved there that if $K \subseteq \R^3$ is a smooth surface, then
\begin{equation}\label{eigen}
\mathcal{V}_R(K;B(x,r)) \approx \frac{2\pi}{3}R^3r^2\bigg(u(x)u(x)^\top + \frac{r^2}{4}\sum_{i=1,2}k_i(x)^2P_i(x)P_i(x)^\top\bigg),
\end{equation}
where $B(x,r)$ is the Euclidean ball with midpoint $x\in K$ and radius $r$, $u(x)$ is one of the two surface unit normals at $x\in K$, $P_1(x),P_2(x)$ are the principal directions and $k_1(x),k_2(x)$ the corresponding principal curvatures. 
Hence, the eigenvalues and -directions of the Voronoi covariance measure carry information about local curvatures and normal directions.

Assuming that a compact set $K_0$  approximates $K$, \cite{merigot} suggests to estimate $\mathcal{V}_R(K;\cdot) $ by $\mathcal{V}_R(K_0;\cdot)$.
It is shown in that paper that $\mathcal{V}_R(K_0;\cdot)$ converges to $\mathcal{V}_R(K;\cdot)$ in the bounded Lipschitz metric when $K_0 \to K$ in the Hausdorff metric. 
Moreover, if $K_0$ is a finite set, then the Voronoi covariance measure can be expressed in the form
\begin{equation*}
\mathcal{V}_R(K_0;A)  = \sum_{x\in K_0 \cap A} \int_{B(x,R)\cap V_x(K_0) } (y-x)(y-x)^\top \,dy.
\end{equation*} 
Here, $V_x(K_0)$ is the Voronoi cell of $x$ in the Voronoi decomposition  of $\R^d$ associated with $K_0$.
Thus, the estimator which is used to approximate $\mathcal{V}_R(K;A)$ is easily computed. Given the Voronoi cells of $K_0$, each Voronoi cell contributes with a simple integral. 
Figure \ref{fig} (a) shows the Voronoi cells of a finite set of points on an ellipse.  The Voronoi cells are elongated in the normal direction. This is the intuitive reason why they can be used to approximate \eqref{eigen}. 

The Voronoi covariance measure $\mathcal{V}_R(K;A) $ can be identified with a symmetric 2-tensor. In the present work, we explore how 
 natural extensions of the Voronoi covariance measure can be used to estimate general Minkowski tensors. 
The generalizations of the Voronoi covariance measure, which we will introduce, will be called \emph{Voronoi tensor measures}. {We will then show how the Minkowski tensors can be recovered from these}. When we apply the results to digital images, we will work with full-dimensional sets $K$, and the finite point sample $K_0$ is obtained from the representation $K_0=K\cap a\La$ of a digital image of $K$.  The Voronoi cells associated with $K_0=K\cap a\La$ are sketched in Figure~\ref{fig}~(b).  Taking point samples from $K$ with increasing resolution, convergence results will follow  from an easy generalization of the convergence proof in \cite{merigot}. 

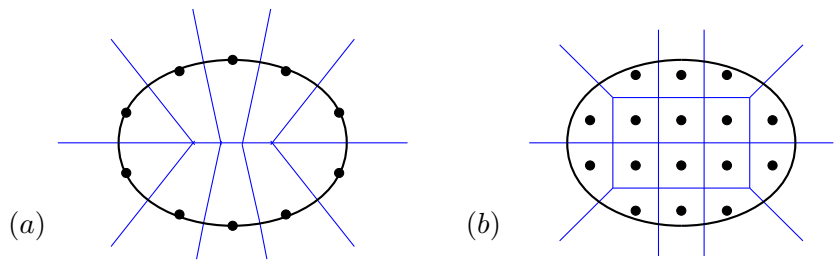
\begin{figure}
\begin{center}
\begin{tikzpicture}
\node at (-0.2,0) {$(a)$};
\node at (5.8,0) {$(b)$};
\draw [fill] (2.5,2.2) circle [radius=0.06];
\draw [blue, thin] (2.125,2.125 ) -- (2.35,1.075);
\draw [blue,thin] (2.125,2.125) -- (1.975,2.875);
\draw [fill] (1.8,2.05) circle [radius=0.06];
\draw [blue, thin] (2.85,2.125 ) -- (2.625,1.075);
\draw [blue,thin] (2.85,2.125) -- (3,2.875);
\draw [fill] (3.2,2.05) circle [radius=0.06];
\draw [blue,thin] (3.55,1.775 ) -- (4.1,2.475);
\draw [blue,thin] (3.55,1.775) -- (3,1.075);
\draw [fill] (3.9,1.5) circle [radius=0.06];
\draw [fill] (3.9,0.7) circle [radius=0.06];
\draw [blue,thin] (3.55,0.425 ) -- (4.1,-0.275);
\draw [blue,thin] (3.55,0.425) -- (3,1.125);
\draw [fill] (3.2,0.15) circle [radius=0.06];
\draw [blue, thin] (2.85,0.075 ) -- (2.625,1.075);
\draw [blue,thin] (2.85,0.075) -- (2.95,-0.442);
\draw [fill] (2.5,0) circle [radius=0.06];
\draw [blue, thin] (2.125,0.075 ) -- (2.35,1.125);
\draw [blue,thin] (2.125,0.075) -- (2.025,-0.442);
\draw [fill] (1.8,0.15) circle [radius=0.06];
\draw [fill] (1.1,0.7) circle [radius=0.06];
\draw [blue,thin] (1.45,0.425 ) -- (0.9,-0.275);
\draw [blue,thin] (1.45,0.425) -- (2,1.125);
\draw [fill] (1.1,1.5) circle [radius=0.06];
\draw [blue,thin] (1.45,1.775 ) -- (0.9,2.475);
\draw [blue,thin] (1.45,1.775) -- (2,1.075);
\draw [blue,thin] (0.2,1.1 ) -- (4.8,1.1);

\draw[thick] (1,1.1) to [out=90,in=180] (2.5,2.2);
\draw[thick] (2.5,2.2) to [out=0,in=90] (4,1.1);
\draw[thick] (4,1.1) to [out=270,in=0] (2.5,0);
\draw[thick] (2.5,0) to [out=180,in=270] (1,1.1);

\draw [fill] (7.2,0.8) circle [radius=0.06];
\draw [fill] (7.2,1.4) circle [radius=0.06];
\draw [fill] (7.8,0.2) circle [radius=0.06];
\draw [fill] (7.8,0.8) circle [radius=0.06];
\draw [fill] (7.8,1.4) circle [radius=0.06];
\draw [fill] (7.8,2) circle [radius=0.06];
\draw [fill] (8.4,0.2) circle [radius=0.06];
\draw [fill] (8.4,0.8) circle [radius=0.06];
\draw [fill] (8.4,1.4) circle [radius=0.06];
\draw [fill] (8.4,2) circle [radius=0.06];
\draw [fill] (9,0.2) circle [radius=0.06];
\draw [fill] (9,0.8) circle [radius=0.06];
\draw [fill] (9,1.4) circle [radius=0.06];
\draw [fill] (9,2) circle [radius=0.06];
\draw [fill] (9.6,0.8) circle [radius=0.06];
\draw [fill] (9.6,1.4) circle [radius=0.06];
\draw [blue,thin] (7.5,0.5) -- (7.5,1.7);
\draw [blue,thin] (8.1,-0.4) -- (8.1,2.6);
\draw [blue,thin] (8.7,-0.4) -- (8.7,2.6);
\draw [blue,thin] (9.3,0.5) -- (9.3,1.7);
\draw [blue,thin] (7.5,0.5) -- (9.3,0.5);
\draw [blue,thin] (6.4,1.1) -- (10.4,1.1);
\draw [blue,thin] (7.5,1.7) -- (9.3,1.7);
\draw [blue,thin] (7.5,0.5) -- (6.8,-0.2);
\draw [blue,thin] (9.3,0.5) -- (10,-0.2);
\draw [blue,thin] (7.5,1.7) -- (6.8,2.4);
\draw [blue,thin] (9.3,1.7) -- (10,2.4);

\draw[thick] (6.9,1.1) to [out=90,in=180] (8.4,2.2);
\draw[thick] (8.4,2.2) to [out=0,in=90] (9.9,1.1);
\draw[thick] (9.9,1.1) to [out=270,in=0] (8.4,0);
\draw[thick] (8.4,0) to [out=180,in=270] (6.9,1.1);

\end{tikzpicture}
\end{center}
\caption{(a). The Voronoi cells of a finite set of points on a surface. (b). A digital image and the associated Voronoi cells.}
\label{fig}
\end{figure}

The paper is structured as follows: In Section~\ref{minkowski}, we recall the definition of Minkowski tensors and the classical 
as well as a local Steiner formula for sets of positive reach. 
In Section~\ref{construction}, we define the  Voronoi tensor measures, discuss how they can be estimated from finite point samples, and explain how the Steiner formula can be used to connect the Voronoi tensor measures with the 
 Minkowski tensors. 
 Section \ref{convergence} is concerned with the convergence of the estimator. The results are specialized to digital images in Section \ref{DI}. 
Finally, the estimator is compared with existing approaches in Section \ref{known}.

\section{Minkowski tensors}\label{minkowski}
We work in Euclidean space $\R^d$ with scalar product $\langle\cdot\,,\cdot\rangle$ and norm $|\cdot|$. The Euclidean 
ball with center $x\in\R^d$ and radius $r\ge 0$ is denoted by $B(x,r)$, and we write $S^{d-1}$ for the unit sphere in $\R^d$. Let $\partial A$ and $\text{int}A$ be the  boundary and the interior of a set $A\subseteq{\mathbb  R}^d$, respectively.
The $k$-dimensional Hausdorff-measure in $\R^d$ is denoted by ${\mathcal H}^k$, $0\le k\le d$. 
Let ${\mathcal C}^d$ be the family of nonempty compact subsets of $\R^d$ and ${\mathcal K}^d\subseteq \mathcal{C}^d$ the subset of nonemtpy compact convex sets. 
For two compact sets $K,M \in{\mathcal C}^d$, we define their \emph{Hausdorff distance} by
\begin{equation*}
d_H(K,M) = \inf\{\eps>0\mid K\subseteq M^\eps, M \subseteq K^\eps\}.
\end{equation*}

Let $\mathbb{T}^p$ denote the space of symmetric $p$-tensors (tensors of rank $p$) over $\R^d$.  
Identifying $\R^d$ with its dual (via the scalar product), a symmetric $p$-tensor defines a symmetric multilinear map $(\R^d)^p\to \R$.
Letting $e_1,\dots,e_d$ be the standard basis in $\R^d$, a tensor $T\in \mathbb{T}^p$ is determined by its coordinates
\begin{equation*}
T_{i_1\dots i_p}=T(e_{i_1},\dots,e_{i_p})
\end{equation*}
with respect to the standard basis, 
for all choices of ${i_1},\dots,{i_p} \in \{1,\dots,d\}$.
We use the norm on $\mathbb{T}^p$ given by
\begin{equation*}
|T|=\sup\big\{|T(v_1,\dots,v_p)| \,\mid \, |v_1|=\dots =|v_p|=1\big\}
\end{equation*}
for $T\in \mathbb{T}^p$. The same definition is used for arbitrary tensors of rank $p$. 

The symmetric tensor product of $y_1,\ldots, y_m\in \mathbb{R}^{d}$ 
	is given by the symmetrization $y_1\odot\cdots\odot y_m=(m!)^{-1}\sum \otimes_{i=1}^m y_{\sigma(i)}$, where the sum extends over all permutations $\sigma$ of $\{1,\ldots,m\}$ and $\otimes$ is the usual tensor product.  We write $x^r$ for the $r$-fold tensor product of $x\in \R^d$. For two symmetric tensors of the form $T_1=y_1 \odot \cdots \odot y_r$ and $T_2=y_{r+1} \odot \cdots \odot y_{r+s}$, 
	where $y_1, \ldots , y_{r+s} \in\R^d$, the symmetric tensor product $T_1\odot T_2$ of $T_1$ and $T_2$, which we often abbreviate by $T_1T_2$, is the symmetric 
	tensor product of $y_1, \ldots, y_{r+s} $. This is extended to general symmetric tensors $T_1$ and $T_2$ by linearity. 
	Moreover, it follows from the preceding definitions that 
	$$
	|y_1\odot\cdots\odot y_m|\le |y_1|\cdots |y_m|,
	$$
$y_1,\ldots, y_m\in \mathbb{R}^{d}$. 

For any compact set $K\subseteq \R^d$, we can define an element of $\mathbb{T}^r$ called the \emph{$r$th volume tensor}
\begin{equation*}
\Phi_{d}^{r,0}(K) = \frac{1}{r!} \int_{K} x^r \,dx.
\end{equation*}
For $s\geq 1$ we define $\Phi_{d}^{r,s}(K)=0$. Some of the volume tensors have well-known physical interpretations. For instance, $\Phi_{d}^{0,0}(K)$ is the usual volume of $K$, $\Phi_{d}^{1,0}(K)$ is up to normalization the center of gravity, and $\Phi_{d}^{2,0}(K)$ is closely related to the tensor of inertia. All three tensors together can be used to find the best approximating ellipsoid of a particle \cite{ziegel}.  The sequence of all volume tensors $(\Phi_{d}^{r,0}(K))_{r=0}^\infty$ determines the compact set $K$ uniquely.
	For convex sets in the plane even the following stability result \cite[Remark 4.4.]{JuliaAstrid} holds: If $K, L\in {\mathcal K}^2$ are contained in the unit square and have coinciding volume tensors up to rank $r$, then their distance, measured in the symmetric difference metric ${\mathcal H}^2\big((K\setminus L) \cup (L\setminus K)\big)$, is of order $O(r^{-1/2})$ as $r\to \infty$. 

We will now define \emph{Minkowski surface tensors}. These can also be used to characterize the shape of an object or the structure of a material as in \cite{beisbart,kapfer}. They require stronger regularity assumptions on $K$. Usually, like in \cite[Section 5.4.2]{schneider}, the set $K$ is assumed to be convex. However, as Minkowski tensors are tensor-valued integrals with respect to the generalized curvature measures (also called support measures) of $K$, they can be defined whenever the latter are available. We will use this to define Minkowski tensors for sets of positive reach.

First, we recall the definition of a set of positive reach and explain how curvature measures of such sets are determined 
(see \cite{Federer59,zahle}). For a compact set $K\in {\mathcal C}^d$, we let $d_K(x)$ denote the distance from $x\in \R^d$ to $K$. Then, for $R\ge 0$,  $K^R=\{x\in \R^d \mid d_K(x)\leq R\}$ is the $R$-parallel set of $K$. The \emph{reach} $\reach(K)$ of $K$ is defined as the supremum over all $R\geq 0$ such that for all $x\in \R^d$ with $d_K(x)<R$ there is a unique closest point $p_K(x)$ in $K$.  We say that $K$ has positive reach if $\reach(K)>0$. Smooth surfaces (of class $C^{1,1}$) are examples of sets of positive reach, and compact convex sets are characterized by having infinite reach. By definition, the map $p_K$ is defined everywhere on $K^R$ if $R<\reach(K)$. 
Let $K\subseteq \R^d$ be a (compact) set of positive reach.  
The  (global) Steiner formula for sets with positive reach states that for all $R<\reach(K)$ the $R$-parallel volume of $K$ is a polynomial, that is,
\begin{align}\label{gloSt}
\Ha^d( K^R){}&=  \sum_{k=0}^d \kappa_{d-k} R^{d-k} \Phi_{k}^{0,0}(K).
\end{align}
 Here $\kappa_j$ is the volume of the unit ball in  $\R^j$ and the numbers 
$\Phi^{0,0}_0(K),\ldots, \allowbreak \Phi_d^{0,0}(K)$ are the so-called \emph{intrinsic volumes} of $K$. They are special cases of the Minkowski tensors to be defined below. Some of them have well-known interpretations. As mentioned, $\Phi^{0,0}_d(K)$ is the volume of $K$. Moreover, $2\Phi^{0,0}_{d-1}(K)$ is the surface area, $\Phi^{0,0}_{d-2}(K)$ is proportional to the total mean curvature, and $\Phi^{0,0}_0(K)$ is the Euler characteristic of $K$. For convex sets, \eqref{gloSt} is the classical Steiner formula which holds for all $R\ge 0$. 

Z\"ahle \cite{zahle} showed that a local version of \eqref{gloSt} can be established giving rise to  the \emph{generalized curvature measures} $\Lambda_k(K;\cdot)$ of $K$, for $k=0,\dots,d-1$. An extension to general closed sets is 
 considered in \cite{last}. The generalized curvature measures (also called support measures) 
are measures on $\Sigma = \R^d\times S^{d-1}$. They are determined by the following {\em local} Steiner formula which holds for all $R < \reach(K)$ and all Borel set $B\subseteq \Sigma$:
\begin{equation}\label{clasSteiner}
\Ha^d\left(\left\{x\in K^R \backslash K \mid \Big(p_K(x), \tfrac{x-p_K(x)}{|x-p_K(x)|}\Big)\in B\right\}\right) = \sum_{k=0}^{d-1} R^{d-k} \kappa_{d-k} \Lambda_k(K;B).
\end{equation}
The coefficients $\Lambda_k(K;B)$ on the right side of  \eqref{clasSteiner} are signed Borel measures $\Lambda_k(K;\cdot)$ evaluated on $B\subseteq\Sigma$. These measures are called the {\em generalized curvature measures}  of $K$. 
 Since the pairs of points in $B$ on the left side of \eqref{clasSteiner} always consist of a boundary point of $K$ and an outer unit normal of $K$ at that point, each of the measures $\Lambda_k(K,\cdot)$ is concentrated on the set of all such pairs. For this reason, the  generalized curvature measures $\Lambda_k(K;\cdot)$, $k\in\{0,\ldots,d-1\}$, are also called {\em support measures}. They describe the local boundary behavior of the part of $\partial K$  that consists of points $x$ with an outer unit normal $u$ such that $(x,u)\in B$. A description of the generalized curvature measures $\Lambda_k(K,\cdot)$  by means 
of generalized curvatures living on the normal bundle of $K$ was first given in \cite{zahle} (see also \cite[\S 2.5 and p.~217]{schneider} and the references given there).   
 	 The total measures $\Lambda_k(K,\Sigma)$ are the intrinsic volumes. 

Based on the generalized curvature measures, for  every $k\in\{0,\dots,d-1\}$, $r,s\geq 0$ and every set $K\subseteq\R^d$ with positive reach, we define the {\em Minkowski tensor}    
\begin{equation*}
\Phi_{k}^{r,s}(K) = \frac{1}{r!s!}\frac{\omega_{d-k}}{\omega_{d-k+s}}\int_{\Sigma} x^r u^{s} \Lambda_k(K;d(x,u)) 
\end{equation*}
in  $\mathbb{T}^{r+s}$. 
Here $\omega_k$ is the surface area of the unit sphere $S^{k-1}$ in $\R^k$. 
More information on Minkowski tensors can for instance be found in \cite{hug,mcmullen,schuster,KVJLNM}.  
As in the case of volume tensors, the Minkowski tensors carry strong information on the underlying set. 
	For instance, already the sequence $(\Phi_{1}^{0,s}(K))_{s=0}^\infty$ determines any $K\in {\mathcal K}^d$ up to a translation. A stability result also holds: if $K$ and $L$ are both contained in a fixed ball and have the same tensors $\Phi_{1}^{0,s}$ of 
	rank $s\le s_0$, then a translation of $K$ is close to $L$ in the Hausdorff metric and the distance is $O(s_0^{-\beta})$ as $s_0\to \infty$ for any $0<\beta<3/(n+1)$; see \cite[Theorem 4.9]{AstridMarkus}.

One can define \emph{local Minkowski tensors} in a similar way (see \cite{HS14}). For a Borel set $B\subseteq \Sigma$, for $k\in\{0,\dots,d-1\}$, $r,s\geq 0$ and a set $K\subseteq\R^d$ with positive reach, we put 
\begin{equation*}
\Phi_{k}^{r,s}(K;B) = \frac{1}{r!s!}\frac{\omega_{d-k}}{\omega_{d-k+s}}\int_{B} x^r u^{s} \,\Lambda_k(K;d(x,u))
\end{equation*}
and, for a Borel set $A \subseteq \R^d$, 
\begin{equation*}
\Phi_{d}^{r,0}(K;A) = \frac{1}{r!} \int_{K\cap A} x^r \,dx.
\end{equation*}
In order to avoid a distinction of cases, we  also write $\Phi_{d}^{r,0}(K;A\times S^{d-1})$ instead of $\Phi_{d}^{r,0}(K;A)$. 
Moreover, we define $\Phi_{d}^{r,s}(K;\cdot)=0$ if $s\ge 1$. 
The local Minkowski tensors can be used to describe local boundary properties. For instance, local 1- and 2-tensors are  used for the detection of sharp edges and corners on surfaces in \cite{clarenz}.  They also carry information about normal directions and principal curvatures  as explained in the introduction. 

We conclude this section with a general remark on continuity properties of the Minkowski tensors. 
	Although the functions $K\mapsto \Phi_{k}^{r,s}(K)$ are continuous when considered in the metric space $(\mathcal{K}^d,d_H)$, they are not continuous on ${\mathcal C}^d$. (For instance, the volume tensors of a finite set are always vanishing, but finite sets can be used to approximate any compact set in the Hausdorff metric.) This is the reason why our approach requires an approximation argument with parallel sets as outlined below. The consistency of our estimator is mainly based on a continuity result for the metric projection map. We quote this result \cite[Theorem 3.2]{chazal} in a slightly different formulation which is symmetric in the two bodies involved. 
	Let  $\|f\|_{L^1(E)}$ be the usual $L^1$-norm of the restriction of $f$ to a Borel  set $E\subseteq \R^d$. 

\begin{proposition}\label{CHAZProp}
	Let $\rho>0$ and let $E\subseteq \R^d$ be a bounded measurable set. Then there is a constant $C_1=C_1\left(d,\diam(E\cup\{0\}),\rho\right)>0$ such that 
	\[
 \|p_K-p_{K_0}\|_{L^1(E)}
	  \le C_1 d_H(K,K_0)^{\frac 12}
	\]
	for all $K,K_0\in {\mathcal C}^d$ with $K,K_0\subseteq B(0,\rho)$. 
\end{proposition}
\begin{proof}
	Let $E'$ be the convex hull of $E$ and observe that 
	\begin{equation*}
\|p_K-p_{K_0}\|_{L^1(E)} \leq  \|p_K-p_{K_0}\|_{L^1(E')}.
\end{equation*}	
    It is shown in \cite[Lemma 3.3]{chazal} (see also \cite[Theorem 4.8]{Federer59}) that the map $v_K:\R^d\to\R$ given by $v_K(x)=|x|^2-d_K^2(x)$ is convex and that its gradient coincides almost everywhere with $2p_K$. Since $E'$ has rectifiable boundary, \cite[Theorem~3.5]{chazal} implies 
    that
    \begin{align*}
    \|p_K-p_{K_0}\|_{L^1(E')}
    \le {}& c_1(d) ({\mathcal H}^d(E')+(c_2+\|d_K^2-d_{K_0}^2\|_{\infty,E'}^{\frac 12}){\mathcal H}^{d-1}(\partial E'))\\
		&\times \|d_K^2-d_{K_0}^2\|_{\infty,E'}^{\frac 12}.
    \end{align*}
    Here $c_2=\diam(2p_K(E')\cup 2p_{K_0}(E'))\le 2\diam (K\cup K_0)\le 4\rho$ and the supremum-norm $\|\cdot\|_{\infty,E'}$ on $E'$ can be estimated by 
    \begin{align*}
    \|d_K^2-d_{K_0}^2\|_{\infty,E'}&\le 2\diam(E'\cup K\cup K_0) \|d_K-d_{K_0}\|_{\infty,E'} 
    \\&\le 2\left[\diam(E'\cup\{0\})+2\rho\right]d_H(K,K_0). 
    \end{align*}
		Moreover, intrinsic volumes are increasing on the class of convex sets, so
		\begin{align*}
		\mathcal{H}^d(E'){}&\leq \Ha^d(B(0, \diam(E'\cup \{0\})))\\
		{\mathcal H}^{d-1}(\partial E') {}&\leq \Ha^{d-1}(\partial B(0, \diam(E'\cup \{0\}))).
		\end{align*}
    Together with the trivial estimate $d_H(K,K_0)\le2\rho$ and with the equality $\diam(E\cup\{0\})=\diam(E'\cup\{0\})$, this yields the claim. 
	\end{proof}
	
	The authors of \cite{chazal} argue that the exponent $1/2$ in Proposition \ref{CHAZProp} is best possible.

\section{Construction of the estimator} \label{construction}
In Section \ref{VTM} below, we define the Voronoi tensor measures and show how the Minkowski tensors can be obtained from these. We then explain  in Section \ref{finite} how the Voronoi tensor measures can be estimated from finite point samples. As a special case, we obtain estimators for all intrinsic volumes. This is detailed in Section \ref{intvol}. 

\subsection{The Voronoi tensor measures} \label{VTM}
Let $K$ be a compact set. Here and in the following subsections, we let $r,s\in\N_0$ and $R\ge 0$. 
Define the $ \mathbb{T}^{r+s}$-valued measures $\mathcal{V}_{R}^{r,s}(K;\cdot)$ given on a Borel set $A\subseteq \R^d$ by
\begin{equation}\label{star}
\mathcal{V}_{R}^{r,s}(K;A) = \int_{K^R }\mathds{1}_A(p_K(x)) \,p_K(x)^r(x-p_K(x))^s \, dx.
\end{equation}
When $K$ is a smooth surface, $\mathcal{V}_{R}^{0,2}(K;\cdot)$ {corresponds to} the Voronoi covariance measure in \cite{merigot}. We will refer to the measures defined in  \eqref{star} as the \emph{Voronoi tensor measures}. 
Note that if $f:\R^d \to \R$ is a bounded Borel function, then
\begin{equation}\label{integralf}
\int_{\R^d} f(x) \,\mathcal{V}_{R}^{r,s}(K;dx) = \int_{K^R}f(p_K(x))\,p_K(x)^r(x-p_K(x))^s \, dx \in \mathbb{T}^{r+s}.
\end{equation}

Suppose now that $K$  has positive reach with $\reach(K)>R$. Then a special case of the generalized Steiner formula derived in \cite{last} (or an extension of \eqref{clasSteiner})  implies the following version of the local Steiner formula for the Voronoi tensor measures: 
\begin{align}\nonumber
\mathcal{V}_{R}^{r,s}(K;A) {}&= \sum_{k=1}^{d} \omega_{k} \int_{\Sigma} \int_{0}^R \mathds{1}_{A}(x) t^{s+k-1} x^r u^s \, dt\,  \Lambda_{d-k}(K;d(x,u))\nonumber\\
&\qquad +\mathds{1}_{\{s = 0\}}\int_{K\cap A} x^r \,dx\nonumber\\
&= r!s! \sum_{k=0}^d \kappa_{k+s} R^{s+k} \Phi_{d-k}^{r,s}(K;A\times S^{d-1}),\label{steiner}
\end{align}
where $A\subseteq {\mathbb R}^d$ is a Borel set.  
 In particular, the total measure is
\begin{equation*}
\mathcal{V}_{R}^{r,s}(K)=\mathcal{V}_{R}^{r,s}(K;\R^d) = r!s!\sum_{k=0}^d \kappa_{k+s}  R^{s+k}  \Phi_{d-k}^{r,s}(K) .
\end{equation*}
Note that the special case $r=s=0$ is the Steiner formula \eqref{gloSt} for sets with positive reach.

Equation \eqref{steiner}, used for different parallel distances $R$, can be solved for the Minkowski tensors.
More precisely, choosing $d+1$ different values $0<R_0<\ldots <R_d<\reach(K)$ for $R$, we obtain a system of $d+1$ linear equations:
\begin{align}\label{matrixeq}
\begin{pmatrix}
\mathcal{V}_{R_0}^{r,s}(K;A)\\
\vdots
\\
\mathcal{V}_{R_d}^{r,s}(K;A)
\end{pmatrix}
=r!s!
\begin{pmatrix}\kappa_s R_0^{s} & \dots & \kappa_{s+d}R_0^{s+d} \\
\vdots & & \vdots
\\
\kappa_sR_{d}^{s} & \dots & \kappa_{s+d}R_{d}^{s+d}
\end{pmatrix}
\begin{pmatrix}\Phi_{d}^{r,s}(K;{A\times S^{d-1}})\\
\vdots
\\
\Phi_{0}^{r,s}(K;{A\times S^{d-1}})
\end{pmatrix}.
\end{align}
Since the Vandermonde-type matrix 
		\begin{align}\label{matrixA}
		A_{R_0,\ldots,R_d}^{r,s}
		=
		{r!s!}
		\begin{pmatrix}
		\kappa_s R_0^{s} & \dots & \kappa_{s+d}R_0^{s+d} \\
		\vdots & & \vdots
		\\
		\kappa_s R_{d}^{s} & \dots & \kappa_{s+d}R_{d}^{s+d}
		\end{pmatrix}\in \R^{(d+1)\times(d+1)}
		\end{align}
in \eqref{matrixeq} is invertible, the system can be solved for the tensors, and thus we get
\begin{align}\label{matrix}
\begin{pmatrix}{\Phi}_{d}^{r,s}(K;A{\times S^{d-1}})\\
\vdots
\\
{\Phi}_{0}^{r,s}(K;{A{\times S^{d-1}}})
\end{pmatrix}
=\left(A_{R_0,\ldots,R_d}^{r,s}\right)^{-1}
\begin{pmatrix}
\mathcal{V}_{R_0}^{r,s}(K;A)\\
\vdots
\\
\mathcal{V}_{R_d}^{r,s}(K;A)
\end{pmatrix}.
\end{align}
If $s>0$, then ${\Phi}_{d}^{r,s}(K;A\times S^{d-1})=0$ by definition, so we may omit one of the equations in the system \eqref{matrixeq}.

\subsection{Estimation of Minkowski tensors}\label{finite}
Let $K$ be a compact set of positive reach.  Suppose that we are given a compact set $K_0$ that is close to $K$ in the Hausdorff metric. In the applications we have in mind, $K_0$ is a finite subset of $K$, but this is not necessary for the algorithm to work. Based on $K_0$, we  want to estimate the local Minkowski tensors of $K$. We do this by approximating $\mathcal{V}_{R_k}^{r,s}(K;A)$ in Formula \eqref{matrix} by $\mathcal{V}_{R_k}^{r,s}(K_0;A)$, for $k=0,\dots,d$ and $A\subseteq \R^d$ a Borel set. 
This leads to the following set of estimators for $\Phi_k^{r,s}(K;A\times S^{d-1})$, $k\in\{0,\ldots,d\}$:
\begin{align}
\begin{pmatrix}\hat{\Phi}_{d}^{r,s}(K_0;A\times S^{d-1})\\
\vdots
\\
\hat{\Phi}_{0}^{r,s}(K_0;A\times S^{d-1})
\end{pmatrix}
=\left(A_{R_0,\ldots,R_d}^{r,s}\right)^{-1}
\begin{pmatrix}
\mathcal{V}_{R_0}^{r,s}(K_0;A)\\
\vdots
\\
\mathcal{V}_{R_d}^{r,s}(K_0;A)
\end{pmatrix}\label{defEst}
\end{align}
with $A_{R_0,\ldots,R_d}^{r,s}$ given by \eqref{matrixA}. Setting  $A=\R^d$ in \eqref{defEst}, we obtain estimators 
\[
\hat{\Phi}_{k}^{r,s}(K_0)=\hat{\Phi}_{k}^{r,s}(K_0;\R^d\times S^{d-1})
\]	
of the {intrinsic volumes}.
Note that this approach requires an estimate for the reach of $K$ because we need to choose $0<R_0<\dots<R_d <\reach(K)$. 
The idea to invert the Steiner formula is not new.  It was used in \cite{chazal} to approximate curvature measures of sets of positive reach.  In \cite{spodarev} and \cite{jan} it was used to estimate intrinsic volumes but without proving convergence for the resulting estimator.    

We now consider the case where $K_0$ is finite. Let
\begin{equation*}
V_x(K_0)=\{y\in \R^d \mid p_{K_0}(y)=x\}
\end{equation*}
denote the Voronoi cell of  $x\in K_0$ with respect to the set $K_0$. Since $\R^d$ is the union of the 
finitely many Voronoi cells of  $K_0$, 
it follows that $K^R_0$ is the union of the $R$-bounded parts $B(x,R)\cap V_x(K_0)$, $x\in K_0$, of the  Voronoi cells 
$V_x(K_0)$, $x\in K_0$,  which have pairwise disjoint interiors. Thus \eqref{star} simplifies to 
\begin{equation}\label{algorithm}
\mathcal{V}_{R}^{r,s}(K_0;A)= \sum_{x\in K_0\cap A } x^r \int_{B(x,R)\cap V_x(K_0)}  (y-x)^s \, dy.
\end{equation}
Like the Voronoi covariance measure, the Voronoi tensor measure $\mathcal{V}_{R}^{r,s}(K_0;A)$ is a sum of simple contributions from the individual Voronoi cells.

An example of a Voronoi decomposition associated with a digital image is sketched in Figure~\ref{redblue}. The original set $K$ is the disk bounded by the inner black circle, and the disk bounded by the outer black circle is its $R$-parallel set $K^R$. The finite point sample is $K_0 = K \cap \Z^2$, which is shown as the set of red dots in the picture, and the red curve is the boundary of its $R$-parallel set. The Voronoi cells of $K_0$ are indicated by blue lines. The $R$-bounded part of one of the Voronoi cells is the part that is cut off by the red arc. 

\begin{figure}
\begin{equation*}
\begin{tikzpicture}[scale=1.4]

\draw[red,fill] (2,0.25) circle [radius=0.025];
\draw[red,fill] (2.4,0.25) circle [radius=0.025];
\draw[red,fill] (2.8,0.25) circle [radius=0.025];
\draw[red,fill] (1.6,0.25) circle [radius=0.025];
\draw[red,fill] (1.2,0.25) circle [radius=0.025];
\draw[red,fill] (2,0.65) circle [radius=0.025];
\draw[red,fill] (2.4,0.65) circle [radius=0.025];
\draw[red,fill] (1.6,0.65) circle [radius=0.025];

\draw[red,fill] (2,-0.15) circle [radius=0.025];
\draw[red,fill] (2.4,-0.15) circle [radius=0.025];
\draw[red,fill] (2.8,-0.15) circle [radius=0.025];
\draw[red,fill] (1.6,-0.15) circle [radius=0.025];
\draw[red,fill] (1.2,-0.15) circle [radius=0.025];
\draw[red,fill] (2,-0.55) circle [radius=0.025];
\draw[red,fill] (2.4,-0.55) circle [radius=0.025];
\draw[red,fill] (1.6,-0.55) circle [radius=0.025];
\draw[red,fill] (2.8,-0.55) circle [radius=0.025];
\draw[red,fill] (1.2,-0.55) circle [radius=0.025];
\draw[red,fill] (2,-0.95) circle [radius=0.025];
\draw[red,fill] (2.4,-0.95) circle [radius=0.025];
\draw[red,fill] (1.6,-0.95) circle [radius=0.025];

\draw [red,thin] (1.8,1.45) arc [radius=0.8, start angle=75, end angle= 157]; 
\draw [red,thin] (0.85,1) arc [radius=0.8, start angle=115, end angle= 195]; 
\draw [red,thin] (3.13,0.99) arc [radius=0.8, start angle=22, end angle= 104]; 
\draw [red,thin] (3.55,0.05) arc [radius=0.8, start angle=-15, end angle= 65]; 
\draw [red,thin] (2.2,1.45) arc [radius=0.8, start angle=75, end angle= 105]; 

\draw [red,thin] (0.425,0.05) arc [radius=0.8, start angle=165, end angle= 195]; 
\draw [red,thin] (3.55,-0.35) arc [radius=0.8, start angle=-15, end angle= 15]; 

\draw [red,thin] (1.795,-1.63) arc [radius=0.8, start angle=289, end angle= 211]; 
\draw [red,thin] (0.87,-1.29) arc [radius=0.8, start angle=245, end angle= 165]; 
\draw [red,thin] (3.115,-1.26) arc [radius=0.8, start angle=-28, end angle= -105]; 
\draw [red,thin] (3.55,-0.35) arc [radius=0.8, start angle=15, end angle= -65]; 
\draw [red,thin] (2.2,-1.64) arc [radius=0.8, start angle=-75, end angle= -105]; 
\draw[blue] (0.1,0.05) --(3.9,0.05); 
\draw[blue] (0.1,-0.35) --(3.9,-0.35); 
\draw[blue] (1.4,0.45) --(2.6,0.45); 
\draw[blue] (1.4,-0.75) --(2.6,-0.75); 
\draw[blue] (1.4,0.05) --(1.4,0.45) --(0.2,1.65);
\draw[blue] (1.4,0.05) --(1.4,-0.75) --(0.2,-1.95);
\draw[blue] (1.4,0.45) --(2.6,0.45); 

\draw[blue] (1.8,0.05) --(1.8,1.85); 
\draw[blue] (1.8,0.05) --(1.8,-2.15); 
\draw[blue] (2.2,0.05) --(2.2,1.85); 
\draw[blue] (2.2,0.05) --(2.2,-2.15);
\draw[blue] (2.6,0.05) --(2.6,0.45) --(3.8,1.65);
\draw[blue] (2.6,0.05) --(2.6,-0.75) --(3.8,-1.95);
\draw [black,very thick] (3,-0.15) arc [radius=1, start angle=0, end angle= 360];
\draw [black,thin] (3.8,-0.1) arc [radius=1.8, start angle=0, end angle= 360];

\end{tikzpicture}
\end{equation*}
\caption{The Voronoi decomposition (blue lines) and $R$-parallel set (red curve) associated with a digital image.}
\label{redblue}
\end{figure}
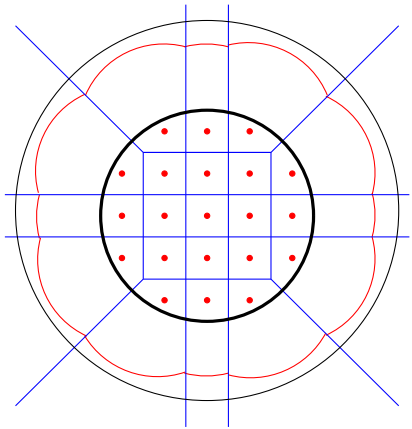

\subsection{The case of intrinsic volumes}\label{intvol}
Recall that $\Phi_k^{0,0}(K)=\Lambda_k(K;\R^d)$ is the $k$th intrinsic volume. Thus, Section \ref{finite} provides estimators for all intrinsic volumes as a special case. This case is particularly simple. The measure $\mathcal{V}_{R}^{0,0}(K;A)$ is simply the volume of a local parallel set
\begin{align*}
\mathcal{V}_{R}^{0,0}(K;A){}&=\Ha^d\left(\{  x \in  K^R\mid p_K(x) \in A\}\right),\\
\mathcal{V}_{R}^{0,0}(K){}&=\Ha^d( K^R).
\end{align*}
In particular, if $K\subseteq \R^d$ is a compact set with $\reach(K)>R$, then 
Equation~\eqref{steiner} reduces to the usual local Steiner formula
\begin{align*}
\Ha^d( \{ x \in  K^R\mid p_K(x) \in A\})&=  \sum_{k=0}^d \kappa_{k} R^{k} \Lambda_{d-k}(K;A\times S^{d-1}),
\end{align*}
and to the (global) Steiner formula \eqref{gloSt} if $A=\R^d$.

In this case, our algorithm approximates the parallel volume $\Ha^d(K^R)$ by $\Ha^d(K_0^R)$. In the example in Figure \ref{redblue}, this corresponds to approximating the volume of the larger black disk by the volume of the region bounded by the red curve. This volume is again the sum of the volumes of the regions bounded by the red and blue curves. In other words, it is the sum of volumes of the $R$-bounded Voronoi cells 
on the right-hand side of the equation
\begin{equation*}
\mathcal{V}_{R}^{0,0}(K_0;A)= \sum_{x\in K_0\cap A } \Ha^d(B(x,R) \cap V_x(K_0)).
\end{equation*}
 
\subsection{Estimators for general local Minkowski tensors}\label{general1}
In Section \ref{finite} we have only considered estimators for local tensors of the form $\Phi_k^{r,s}(K;A \times S^{d-1})$, where $K\subseteq\R^d$ is a set with positive reach. The natural way to estimate $\Phi_k^{r,s}(K;B)$, for a measurable set $B\subseteq \Sigma $, would be to copy the idea in Section \ref{finite}  with $\mathcal{V}_{R}^{r,s}(K;A)$ replaced by the following generalization of the Voronoi tensor measures,
\begin{equation}\label{baddef}
\mathcal{W}_{R}^{r,s}(K;B) =
 \int_{K^R \backslash K}\mathds{1}_B(p_K(x),u_K(x))p_K(x)^r (x-p_K(x))^s \, dx,
\end{equation}
where $u_K(x) = ({x-p_K(x)})/{|x-p_K(x)|}$ estimates the normal direction.  
Of course, this definition works for any $K\in\mathcal{C}^d$. 
Moreover, we could define estimators related to \eqref{baddef} whenever we have a set $K_0$ which approximates $K$. However, even if 
$K$ has positive reach, the map $x\mapsto u_K(x)$ is not Lipschitz on $K^R\backslash K$, and therefore 
 the convergence results in Section \ref{convergence} will not work with this definition. Since the map $x\mapsto u_K(x)$ is Lipschitz on $K^R\backslash K^{R/2}$, it is natural to proceed as follows. 
For any $K\in\mathcal{C}^d$, we define
\begin{align}\label{modify}
\overline{\mathcal{V}}_{R}^{r,s}(K;B) {}&=
 \int_{K^R \backslash K^{R/2}}\mathds{1}_B(p_K(x),u_K(x))p_K(x)^r (x-p_K(x))^s \, dx. 
\end{align}
Note that 
\begin{equation}\label{Wdifference}
\overline{\mathcal{V}}_{R}^{r,s}(K;\cdot)=\mathcal{W}_{R}^{r,s}(K;\cdot)-\mathcal{W}_{R/2}^{r,s}(K;\cdot),
\end{equation} 
where $\mathcal{W}_{R}^{r,s}(K;\cdot)$ is defined as in \eqref{baddef}. We will not use the notation $\mathcal{W}_{R}^{r,s}(K;\cdot)$ in the following.
If $K$ has positive reach and $0<R<\text{reach}(K)$, then the generalized Steiner formula yields 
\begin{align*}
{\overline{\mathcal{V}}_{R}^{r,s}(K;B)}
&= r!s! \sum_{k=1}^d \kappa_{s+k} R^{s+k} (1-2^{-(s+k)}){ {\Phi}_{d-k}^{r,s}(K;B)}.
\end{align*}
Again, choosing $0<R_1<\ldots<R_d<\text{reach}(K)$, we can recover the Minkowski tensors from
\begin{align*}
\begin{pmatrix}{\Phi}_{d-1}^{r,s}(K;B)\\
\vdots
\\
{\Phi}_{0}^{r,s}(K;B)
\end{pmatrix}
=
\left(
\overline{A}_{R_1,\ldots,R_d}^{r,s}
\right)^{-1}
\begin{pmatrix}
\overline{\mathcal{V}}_{R_1}^{r,s}(K;B)\\
\vdots
\\
\overline{\mathcal{V}}_{R_d}^{r,s}(K;B)
\end{pmatrix}
\end{align*}
where 
\[
\overline {A}_{R_1,\ldots,R_d}^{r,s}=
\frac{1}{r!s!}
\begin{pmatrix}
\kappa_{s+1} (1-2^{-(s+1)}) R_1^{s+1} & \dots & \kappa_{s+d}(1-2^{-(s+d)})R_1^{s+d} \\
\vdots & & \vdots
\\
\kappa_{s+1} (1-2^{-(s+1)}) R_{d}^{s+1} & \dots & \kappa_{s+d}(1-2^{-(s+d)})R_{d}^{s+d}
\end{pmatrix}
\]
is a regular matrix.
Using this, we can define estimators for  ${\Phi}_{k}^{r,s}(K;B)$, for $0\le k\leq d-1$, by
\begin{align*}
\begin{pmatrix}\overline{\Phi}_{d-1}^{r,s}(K_0;B)\\
\vdots
\\
\overline{\Phi}_{0}^{r,s}(K_0;B)
\end{pmatrix}
=
\left(
\overline{A}_{R_1,\ldots,R_d}^{r,s}
\right)^{-1}
\begin{pmatrix}
\overline{\mathcal{V}}_{R_1}^{r,s}(K_0;B)\\
\vdots
\\
\overline{\mathcal{V}}_{R_d}^{r,s}(K_0;B)
\end{pmatrix},
\end{align*}
where $K_0$ is a compact set which approximates $K$.  
Convergence of these modified estimators will be discussed in Section \ref{convergence}.

The estimators $\overline{\Phi}_{k}^{r,s}$ can be used to approximate local tensors of the form $\Phi_k^{r,s}(K;B)$ where the set $B\subseteq \Sigma$ involves normal directions. Thus,  they are more general than $\hat{\Phi}_{k}^{r,s}$. However, \eqref{Wdifference} shows that estimating $\overline{\mathcal{V}}_{R}^{r,s}(K;B)$ requires an approximation of two parallel sets, rather than one. We therefore expect more severe numerical errors for $\overline{\Phi}_{k}^{r,s}$.

\section{Convergence properties}\label{convergence}
In this section we prove the main convergence results. This is an immediate generalization of \cite[Theorem 5.1]{merigot}.

\subsection{The convergence theorem}
For a bounded Lipschitz function $f:\R^d \to \R$, we let $|f|_\infty$ denote the usual supremum norm,
\begin{equation*}
|f|_L = \sup \bigg\{ \frac{|f(x)-f(y)|}{|x-y|} \mid x\neq y\bigg\}
\end{equation*}
the Lipschitz semi-norm, and
\begin{equation*}
|f|_{bL}=|f|_L + |f|_\infty
\end{equation*}
the bounded Lipschitz norm. 
Let $d_{bL} $ be the bounded Lipschitz metric on the space of bounded $\mathbb{T}^p$-valued Borel measures on $\R^d$. 
For any  two such measures $\mu$ and $\nu$ on $\R^d$, the distance with respect to $d_{bL}$ is defined by
\begin{equation*}
d_{bL}(\mu,\nu) = \sup \bigg\{\bigg|\int f \, d\mu - \int f \, d\nu\bigg| \mid |f|_{bL} \leq 1\bigg\},
\end{equation*}
where the supremum extends over all bounded Lipschitz functions $f:\R^d\to \R$ with $ |f|_{bL} \leq 1$.  
The following theorem shows that the map																																													
\begin{equation*}
K \mapsto \mathcal{V}_{R}^{r,s}(K;\cdot)
\end{equation*}
is H\"{o}lder continuous with exponent $ \frac{1}{2}$ with respect to the Hausdorff metric on $\mathcal{C}^d$ (restricted to compact subsets of a fixed ball) and the bounded Lipschitz metric. 
In the proof, we use the symmetric difference $A\Delta B=(A\setminus B)\cup(B\setminus A)$ of sets $A,B\subseteq \R^d$. 
\begin{thm}\label{converge}
Let  $R,\rho>0$ and $r,s\in {\mathbb N}_0$ be given. Then there is a positive  constant $C_2=C_2(d,R,\rho,r,s)$ such that
\begin{align*}
d_{bL}(\mathcal{V}_{R}^{r,s}(K;\cdot),\mathcal{V}_{R}^{r,s}(K_0;\cdot )) \leq C_2 d_H(K,K_0)^{\frac{1}{2}} 
\end{align*}
for all compact sets $K,K_0\subseteq B(0,\rho)$.
\end{thm}

\begin{proof} 
Let $f$ with $|f|_{bL} \leq 1$ be given. Then \eqref{integralf} yields
\begin{align}\nonumber
&\bigg|\int_{\R^d} f(x)\, \mathcal{V}_{R}^{r,s}(K;dx)-\int_{\R^d} f(x)\, \mathcal{V}_{R}^{r,s}(K_0;dx)\bigg|\\
&=\bigg|\int_{K^R }f(p_K(x))\,p_K(x)^r(x-p_K(x))^s \, dx\nonumber\\
&\qquad\qquad\qquad\qquad-\int_{K_0^R }f(p_{K_0}(x))p_{K_0}(x)^r(x-p_{K_0}(x))^s \, dx\bigg|
\nonumber
\\
&\leq {I}+{II},\label{AB}
\end{align}
where $I$ is the integral
\begin{align*}	
\int_{K^R \cap K_0^R }|f(p_K(x))p_K(x)^r(x-p_K(x))^s-f(p_{K_0}(x))\,p_{K_0}(x)^r(x-p_{K_0}(x))^s |\,dx
\end{align*}
and 
\begin{align*}	
II={{\rho}^r}R^s \Ha^{d}(K^R \Delta K_0^R). 
\end{align*}
By \cite[Corollary 4.4]{chazal}, there is a constant $c_1=c_1(d,R,\rho)>0$ such that
\begin{equation}\label{symdif}
\Ha^{d}(K^R \Delta K_0^R) \leq c_1\,d_H(K,K_0)
\end{equation}
when $d_H(K,K_0)\leq {R}/{2}$. 
Replacing $c_1$ by a possibly even bigger constant, we can ensure that \eqref{symdif} also holds when 
$R/2\le d_H(K,K_0)\leq 2 \rho$.	Hence, 
\begin{equation}\label{B}
II \leq {c_2}\,d_H(K,K_0)^{\frac 12}
\end{equation}	
with some constant $c_2=c_2(d,R,\rho,r,s)>0$.

Using the inequalities (and interpreting empty products as 1)  
\begin{align}\label{product}
\bigg|\bigodot_{i=1}^m y_i - \bigodot_{i=1}^m z_i\bigg|\leq
\bigg|\bigotimes_{i=1}^m y_i - \bigotimes_{i=1}^m z_i\bigg|\leq \sum_{j=1}^m |y_j-z_j|\prod_{i=1}^{j-1} |y_i| \prod_{i=j+1}^m |z_i|,
\end{align}
with $m=r+s$ and the rank-one tensors 
	\[
	\begin{array}{lcl}
    y_1=\ldots=y_r=p_K(x), &\qquad& y_{r+1}=\ldots=y_{r+s}=x-p_K(x),\\
    z_1=\ldots=z_r=p_{K_0}(x), &\qquad& z_{r+1}=\ldots=z_{r+s}=x-p_{K_0}(x),
	\end{array}
	\]
we get 
\begin{align*}	
|f{}&(p_K(x))\, p_K(x)^r(x-p_K(x))^s-f(p_{K_0}(x))\,p_{K_0}(x)^r(x-p_{K_0}(x))^s | \\
&\leq |f(p_K(x))-f(p_{K_0}(x)) | |p_K(x)|^{r} |x-p_{K}(x)|^s \\
 &+|f(p_{K_0}(x))| \sum_{j=1}^r |p_K(x)-p_{K_0}(x)||p_K(x)|^{j-1} |p_{K_0}(x)|^{r-j}|x-p_{K_0}(x)|^s + \\
 &+|f(p_{K_0}(x))| \sum_{j=1}^s |p_K(x)-p_{K_0}(x)||p_K(x)|^{r} |x-p_{K}(x)|^{j-1}|x-p_{K_0}(x)|^{s-j}. 
\end{align*}
Since we assumed that $|f|_{bL}\le 1$, we get
\begin{align}\nonumber
I&\leq (r+s+1)\max\{\rho,1\}^r\max\{R,1\}^s\int_{K^R \cap K_0^R } |p_K(x)-p_{K_0}(x)|\, dx\\
&\leq   c_3\, d_H(K,K_0)^{\frac{1}{2}}.\label{A}
\end{align}
The existence of the constant $c_3=c_3(d,R,\rho,r,s)$ in the last inequality is guaranteed by Proposition \ref{CHAZProp} with $K^R \cap K_0^R$ as the set $E$, because this choice of $E$ satisfies $\diam(E \cup \{0\})\leq 2(\rho + R)$. 
\end{proof}

When $r=s=0$ and $f=1$, the above proof simplifies to Inequality \eqref{symdif} as $I$ vanishes. Hence we obtain the following strengthening of the theorem, which is relevant for the estimation of intrinsic volumes.

\begin{thm}\label{IVconverge}
Let $R,\rho>0$. 
Then there is a constant $C_3=C_3(d,R,\rho)>0$ such that
\begin{equation*}
\Big|\mathcal{V}_{R}^{0,0}(K)-\mathcal{V}_{R}^{0,0}(K_0) \Big|\leq C_3\, d_H(K,K_0)
\end{equation*}
for all compact sets  $K,K_0\subseteq B(0,\rho)$. 
\end{thm}

For local tensors, the proof of Theorem \ref{converge} can also be adapted to show a convergence result. 

\begin{thm}\label{locallip}
Let $r,s\in\N_0$ and $R>0$. 
If $K_i \to K$ with respect to the Hausdorff metric on ${\mathcal C}^d$, as $i\to \infty$, then $\mathcal{V}_{R}^{r,s}(K_i;A)\to \mathcal{V}_{R}^{r,s}(K;A)$ in the tensor norm, for every Borel set $A\subseteq\R^d$ which satisfies 
\begin{equation}\label{4.3exceptional}
\Ha^d(p_K^{-1}(\partial A)\cap K^R)=0.
\end{equation}
\end{thm}
\begin{proof}
Convergence of tensors is equivalent to coordinate-wise convergence.
Hence, it is enough to show that the coordinates satisfy 
$$\mathcal{V}_{R}^{r,s}(K_i;A)_{i_1\dots i_{r+s}}\to \mathcal{V}_{R}^{r,s}(K;A)_{i_1\dots i_{r+s}}\qquad\text{as $i\to\infty$},$$ 
for all choices of indices ${i_1\dots i_{r+s}}$; see the notation at the beginning of Section~\ref{minkowski}. 

We write $T_K(x)=p_K(x)^r(x-p_K(x))^s$. Then
\begin{equation*}
\mathcal{V}_{R}^{r,s}(K;A)_{i_1\dots i_{r+s}}=\int_{K^R} \mathds{1}_A(p_K(x))T_K(x)_{i_1\dots i_{r+s}}\, dx
\end{equation*}
is a signed measure.
Let $T_K(x)_{i_1\dots i_{r+s}}^+$ and $T_K(x)_{i_1\dots i_{r+s}}^-$ denote the positive and negative part of $T_K(x)_{i_1\dots i_{r+s}}$, respectively. Then 
\begin{equation*}
\mathcal{V}_{R}^{r,s}(K;A)^{\pm}_{i_1\dots i_{r+s}}=\int_{K^R} \mathds{1}_A(p_K(x))T_K(x)_{i_1\dots i_{r+s}}^{\pm}\,dx
\end{equation*}
are non-negative measures such that
\begin{equation*}
\mathcal{V}_{R}^{r,s}(K;\cdot)_{i_1\dots i_{r+s}}=\mathcal{V}_{R}^{r,s}(K;\cdot)_{i_1\dots i_{r+s}}^+-\mathcal{V}_{R}^{r,s}(K;\cdot)_{i_1\dots i_{r+s}}^-.
\end{equation*}

The proof of Theorem \ref{converge} can immediately be generalized to show that $\mathcal{V}_{R}^{r,s}(K_i;\cdot)^{\pm}_{i_1\dots i_{r+s}}$ converges to $\mathcal{V}_{R}^{r,s}(K;\cdot)^{\pm}_{i_1\dots i_{r+s}}$  in the bounded Lipschitz norm (as $i\to\infty$), and hence the measures converge weakly. In particular, they converge on every continuity set of 
$\mathcal{V}_{R}^{r,s}(K;\cdot)^{\pm}_{i_1\dots i_{r+s}}$. If $\Ha^d(p_K^{-1}(\partial A)\cap K^R)=0$, then $A$ is such a continuity set. 
\end{proof}

\begin{remark}
Though relatively mild, the condition $\Ha^d(p_K^{-1}(\partial A)\cap K^R)=0$ can be hard to control if $K$ is unknown. It is satisfied if, for instance, $K$ and $A$ are smooth and their boundaries intersect transversely. A special case of this is when $K$ is a smooth surface and $A$ is a small ball centered on the boundary of $K$. This is the case in the application from \cite{merigot} that was described in the introduction. 
Examples where it is not satisfied are when $A=K$ or when $K$ is a polytope intersecting $\partial A$ at a vertex. 
\end{remark}

\begin{remark}\label{Rem4.6new}
Let $f:\R^d\to\R$ be a bounded measurable function. We define 
$$
\mathcal{V}_{R}^{r,s}(K;f):=\int_{\R^d} f(x)\, \mathcal{V}_{R}^{r,s}(K;dx).
$$
Hence $\mathcal{V}_{R}^{r,s}(K;A)=\mathcal{V}_{R}^{r,s}(K;\mathds{1}_A)$ for every Borel set $A\subseteq\R^d$. 
Then, Theorem \ref{locallip} is equivalent to saying that, for all continuous test functions $f:\R^d\to\R$,  
$$
\mathcal{V}_{R}^{r,s}(K_i;f)\to \mathcal{V}_{R}^{r,s}(K;f),\quad \text{as }i\to\infty,
$$
in the tensor norm, 
whenever $K_i \to K$ with respect to the Hausdorff metric on ${\mathcal C}^d$, as $i\to \infty$. 
 Thus, if one is interested in the local behaviour of $\Phi^{r,s}_k(K; \cdot)$ at a neighborhood $A$, like in \cite{merigot}, then
one can study
$$
\Phi^{r,s}_k(K;f):=\int_{\Sigma} f(x)x^ru^s\, \Lambda_k(K;d(x,u)), 
$$
where $f$ is a continuous function with support in $A$. This avoids the extra condition \eqref{4.3exceptional}. 

\end{remark}

As the matrix  $A_{R_0,\ldots,R_d}^{r,s}$ in the definition \eqref{defEst} of $\hat \Phi_k^{r,s}(K_0;A\times S^{d-1})$ does not depend on the set $K_0$, the above results immediately yield a consistency result for the estimation of the Minkowski tensors. We formulate this only for $A=\R^d$.
	\begin{corollary}\label{corNew}
		Let $\rho>0$ and $K$ be a compact subset of $B(0,\rho)$ of positive reach such that $\mathrm{Reach}(K)>R_d>\ldots>R_0>0$. 
		Let $K_0\subseteq B(0,\rho)$ be a compact set. 
		Then there is a  constant  $C_4=C_4(d,R_0,\ldots,R_d,\rho)$ such that 
		\[
		 \left| \hat{\Phi}^{0,0}_k(K_0)-\Phi^{0,0}_k(K)\right|\le C_4\, d_H(K_0,K),
		\]
		for all $k\in\{0,\ldots,d\}$.

		For $r,s\in {\mathbb N}_0$ there is a constant 
		$C_5=C_5(d,R_0,\ldots,R_d,\rho,r,s)$ such that 
		\[
		\left| \hat{\Phi}^{r,s}_k(K_0)-\Phi^{r,s}_k(K)\right|\le C_5\, d_H(K_0,K)^{\frac12},
		\] 
		 for all $k\in\{0,\ldots,d-1\}$. 
		\end{corollary}

Finally, we state the convergence results for the modified estimators	for $\Phi_k^{r,s}(K;B)$, where $B\subseteq \Sigma$ 
is a Borel set, that were defined in Section \ref{general1}.
The map $x\mapsto {x}/{|x|}$ is Lipschitz on $\R^d \backslash \indre({B(0,{R}/{2})})$ with Lipschitz constant ${4}/{R}$, and therefore the mapping $u_K$, which was defined after \eqref{baddef}, satisfies
\begin{equation*}
|u_K(x)-u_{K_0}(x)|\leq \tfrac{4}{R}|p_K(x)-p_{K_0}(x)|,
\end{equation*}
for $x\in (K^R \backslash K^{R/2}) \cap(K_0^R \backslash K_0^{R/2})$. Moreover,
\begin{equation*}
\left(K^R \backslash K^{R/2}\right) \Delta \left(K_0^R \backslash K_0^{R/2}\right) 
\subseteq \left(K^R \Delta K_0^{R}\right) \cup \left(K^{R/2} \Delta K_0^{R/2}\right).
\end{equation*}
Using this, it is straightforward to generalize the proofs of Theorems \ref{converge} and \ref{locallip} to obtain the following  result. 

\begin{thm}\label{convergeloc2}
Let  $R,\rho>0$ and $r,s\in {\mathbb N}_0$ be given. Then there is a positive  constant $C_6=C_6(d,R,\rho,r,s)$ such that
\begin{align*}
d_{bL}(\overline{\mathcal{V}}_{R}^{r,s}(K;\cdot),\overline{\mathcal{V}}_{R}^{r,s}(K_0;\cdot )) \leq C_6 d_H(K,K_0)^{\frac{1}{2}} 
\end{align*}
for all compact sets $K,K_0\subseteq B(0,\rho)$.
\end{thm}

This in turn leads to the next convergence result.

\begin{thm}
Let $r,s\in\N_0$ and $R>0$.  
If $K,K_i\in \mathcal{C}^d$ are compact sets such that $K_i\to K$ in the Hausdorff metric, as $i\to\infty$,  then $\overline{\mathcal{V}}_{R}^{r,s}(K_i;B)$ converges to $\overline{\mathcal{V}}_{R}^{r,s}(K;B)$ in the tensor norm, for any measurable set $B\subseteq \Sigma$ satisfying  
\begin{equation*}
\Ha^d(\{x\in K^R \mid (p_K(x), u_K(x))\in \partial B\})=0.
\end{equation*}
Here $\partial B$ is the boundary of $B$ as a subset of $\Sigma$. 

If $B$ satisfies this condition and $\text{Reach}(K)>R_d$, then
\begin{equation*}
\lim_{i \to 0} \overline{\Phi}_{k}^{r,s}(K_i;B) = {\Phi_{k}^{r,s}(K;B)} . 
\end{equation*}
\end{thm}

\begin{remark}
We can argue as in Remark \ref{Rem4.6new} to see that if $K,K_i\in \mathcal{C}^d$ are compact sets such that $K_i\to K$ in the Hausdorff metric, as $i\to\infty$, then 
$$
\overline{\mathcal{V}}_{R}^{r,s}(K_i;g)\to \overline{\mathcal{V}}_{R}^{r,s}(K;g),\quad \text{as }i\to\infty,
$$
whenever $g:\Sigma\to\R$ is a continuous test function and $\overline{\mathcal{V}}_{R}^{r,s}(K;g)$ is defined similarly as before. 

If $K$ satisfies $\text{Reach}(K)>R_d$, we get $\overline{\Phi}_{k}^{r,s}(K_i;g) \to {\Phi_{k}^{r,s}(K;g)}$, as $i\to\infty$. 
\end{remark}

\section{Application to digital images}\label{DI}
Our main motivation for this paper is the estimation of Minkowski tensors from digital images. Recall that we model a black-and-white digital image of $K\subseteq \R^d$ as the set $K\cap a\La$, where $\La\subseteq \R^d$ is a fixed lattice and $a>0$. We refer to \cite{barvinok02} for 
basic information about lattices.

The lower dimensional parts of $K$ are generally invisible in the digital image. When  dealing with digital images, we will therefore always assume that the underlying set is topologically regular, which means that it is the closure of its own interior. 

In digital stereology,  the underlying object $K$ is often assumed to belong to one of the following two set classes: 
\begin{itemize}
\item
 $K$ is called \emph{$\delta$-regular} if it is topologically regular and the reach of its closed complement ${\rm cl}({\R^d \backslash K})$ and the reach of $K$ itself are both at least $\delta>0$.  This is a kind of smoothness condition on the boundary, ensuring in particular that $\partial K$ is a $C^1$ manifold (see the discussion after Definition 1 in \cite{svane15b}). 
\item  $K$ is called \emph{polyconvex} if it is a finite union of compact convex sets. While convex sets have infinite reach, note that polyconvex sets do generally not have positive reach. Also note that for a compact convex set $K\subseteq\R^d$, the set ${\rm cl}({\R^d \backslash K})$ need not have positive reach.
\end{itemize}
It should be observed that for a compact set $K\subseteq \R^d$ both assumptions imply that the boundary of $K$ is a $(d-1)$-rectifiable set in the sense of \cite{Federer69} (i.e., $\partial K$ is the image of a bounded subset of $\R^{d-1}$ under a Lipschitz map), which is a much weaker property that will be sufficient for the analysis in Section \ref{volten}. 

\subsection{The volume tensors}\label{volten}
Simple and efficient estimators for the volume tensors $\Phi_d^{r,0}(K)$ of a (topologically regular) compact set $K$ are already known
and are usually based on the approximation of $K$ by the union of all pixels (voxels) with midpoint in $K$. This leads to the estimator
\begin{equation*}
\phi_d^{r,0}(K\cap a\La ) = \frac1 {r!}  \sum_{z \in K\cap a\La} \int_{z+aV_0(\La)}x^r\,dx,
\end{equation*}
where $V_0(\La)$ is the Voronoi cell of 0 in the Voronoi decomposition generated by $\La$. 
This, in turn, can be approximated by
\begin{equation*}
\hat{\phi}_d^{r,0}(K\cap a\La ) =  \frac{a^{d}}{r!} \Ha^d\left(V_0(\La)\right) \sum_{z \in K\cap a\La} z^r.
\end{equation*}
When $r\in \{0,1\}$, we even have ${\phi}_d^{r,0}(K\cap a\La )=\hat{\phi}_d^{r,0}(K\cap a\La )$. 

Choose $C>0$  such that $V_0(\La) \subseteq B(0,C)$. Then 
$$
K\Delta \bigcup_{z\in K\cap a\La} (z+aV_0(\La))\subseteq (\partial K)^{ aC}.
$$ 
In fact, if $x\in \left[\bigcup_{z\in K\cap a\La} (z+aV_0(\La))\right]\setminus K$, then there is some $z\in K\cap a\La$ such that 
$x\in z+aV_0(\La)$ and $x\notin K$. Since $z\in K$ and $x\notin K$, we have $[x,z]\cap\partial K\neq\emptyset$. Moreover, $x-z\in aV_0(\La)\subseteq  
B(0,aC)$, and hence $|x-z|\le aC$. This shows that $x\in(\partial K)^{aC}$. Now assume that $x\in K$ and $x\notin (\partial K)^{aC}$. Then $B(x,\rho)\subseteq K$ for some $\rho>aC$. Since $\bigcup_{z\in a\La}(z+aV_0(\La))=\R^d$, there is some $z\in a\La$ such that 
$x\in z+aV_0(\La)$. Hence $x-z\in aV_0(\La)\subseteq B(0,aC)$. We conclude that $z\in B(x,aC)\subseteq K$, therefore $z\in K\cap a\La$ and thus
$x\in \bigcup_{z\in K\cap a\La} (z+aV_0(\La))$. 

Hence
\begin{equation}\label{Oabound}
|{\phi}_d^{r,0}(K\cap a\La ) - {\Phi}_d^{r,0}(K)| \leq \frac{1} {r!}   \int_{(\partial K)^{ aC}}|x|^r \, dx.
\end{equation}
If $\Ha^{d}(\partial K)=0$, then the integral on the right-hand side goes to zero by monotone convergence, so 
\begin{equation}\label{convzero}
\lim_{a\to 0_+}{\phi}_d^{r,0}(K\cap a\La ) ={\Phi}_d^{r,0}(K).
\end{equation}
If $\partial K$ is  $(d-1)$-rectifiable in the sense of \cite[Section 3.2.14]{Federer69}, that is, $\partial K$ is the image of a bounded subset of $\R^{d-1}$ under a Lipschitz map, then $\Ha^{d}(\partial K)=0$. Since $\partial K$ is compact, \cite[Theorem 3.2.39]{Federer69} implies that  $\lim_{a\to 0_+}\Ha^d((\partial K)^{ aC})/a $ exists and equals a fixed multiple of $\Ha^{d-1}(\partial K)$ which is finite. Hence, \eqref{Oabound} shows that the speed of convergence in \eqref{convzero} is $O(a)$ as $a\to 0_+$.

Inequality \eqref{product} yields that $|x^r-z^{r}|\leq aC r(|x|+aC)^{r-1}$ whenever $x\in z+ aV_0(\La)$ and $r\ge 1$. Therefore, 
\begin{align*}
|\hat{\phi}_d^{r,0}(K\cap a\La ) - \phi_d^{r,0}(K\cap a \La)|{}& \leq \frac{aC } {(r-1)!}  \sum_{z \in K\cap a\La} \int_{z+aV_0(\La)}(|x|+aC)^{r-1}\, dx\\
& \leq \frac{aC } {(r-1)!} \int_{K^{aC}} (|x|+aC)^{r-1} \,dx,
\end{align*}
which shows that
\begin{equation*}
\lim_{a\to 0_+}\hat{\phi}_d^{r,0}(K\cap a\La ) ={\Phi}_d^{r,0}(K),
\end{equation*}
provided that $\mathcal{H}^d(\partial K)=0$. If $\partial  K$ is $(d-1)$-rectifiable, then the speed of convergence is of the order $O(a)$. 

Hence, we suggest to simply use the estimators $\hat{\phi}_d^{r,0}(K\cap a\La )$ for the volume tensors. This estimator can be computed much faster and more directly than $\hat{\Phi}_d^{r,0}(K\cap a\La )$. Moreover, it does not require an estimate for the reach of $K$, and it converges for a much larger class of sets than those of positive reach.

\subsection{Convergence for digital images}
For the estimation of the remaining tensors we suggest to use the Voronoi tensor measures. Choosing $K_0=K \cap a\La$ in \eqref{algorithm}, 
we obtain
\begin{equation}\label{algorithm2}
\mathcal{V}_{R}^{r,s}(K\cap a\La ;A)= \sum_{x\in K  \cap a\La \cap A } x^r \int_{B(x,R)\cap V_x(K\cap a\La)}  (y-x)^s \,dy,
\end{equation}
where $A\subseteq\R^d$ is a Borel set. 

To show some convergence results in Corollary \ref{convercor} below, we first note that the digital image converges to the original set in the Hausdorff metric. 

\begin{lemma}\label{dHbounds}
If $K$ is compact and topologically regular, then 
\begin{equation*}
\lim_{a\to 0_+} d_H(K,K\cap a\La) = 0.
\end{equation*}
If $K $ is $\delta$-regular, then $d_H(K,K\cap a\La)$ is of order $O(a)$. The same holds if $K$ is topologically regular and polyconvex. 
\end{lemma}

\begin{proof} Recall from \cite[p.~311]{barvinok02} that  
$
\mu(\La)=\max_{x\in\R^d}\text{dist}(x,\La)
$ 
is well defined and denotes the covering radius of $\La$. 

Let $\eps>0$ be given. 
Since $K$ is compact, there are  points $x_1,\ldots,x_m\in K$ such that 
$$
K\subseteq\bigcup_{i=1}^m B(x_i,\eps).
$$
Using the fact that $K$ is topologically regular, we conclude that there are points $y_i\in\text{int}(K)\cap \text{int}(B(x_i,2\eps))$ for $i=1,\ldots,m$. Hence, there are $\eps_i\in (0,2\eps)$ such that  $ B(y_i,\eps_i)\subseteq K\cap B(x_i,2\eps)$ for $i=1,\ldots,m$. 
Let $0<a<\min\{\eps_i/\mu(\La) \mid i=1,\ldots,m\}$. Since $\eps_i/a>\mu(\La)$ it follows that $a\La \cap B(y_i,\eps_i)\neq\emptyset$, for  
$i=1,\ldots,m$. Thus we can choose $z_i\in a\La \cap B(y_i,\eps_i)\subseteq a\La\cap K$ for  
$i=1,\ldots,m$. By the triangle inequality, we have $|z_i-x_i|\le \eps_i+2\eps\le 4\eps$, and hence $x_i\in (K\cap a\La)+B(0,4\eps)$, for 
$i=1,\ldots,m$. Therefore, $K\subseteq (K\cap a\La) +B(0,5\eps)$ if $a>0$ is sufficiently small. 

Assume that $K$ is $\delta$-regular, for some $\delta>0$. We choose $0<a<\delta/(2\mu(\La))$. Since $a\mu(\La)<\delta/2$, for any $x\in K$ there is a ball $B(y,a\mu(\La))$ of radius $a\mu(\La)$ such that $x\in B(y,a\mu(\La))\subseteq K$. From $a\La\cap B(y,a\mu(\La))\neq\emptyset$ we 
conclude that there is a point $z\in K\cap a\La$ with $|x-z|\le 2a\mu(\La)$. Hence $x\in (K\cap a\La) +B(0,2a\mu(\La))$, and therefore $d_H(K,K\cap a\La)\le 2a\mu(\La)$. 

Finally, we assume that $K$ is topologically regular and polyconvex. Then $K$ is the union of finitely many compact convex sets with interior points. Hence, for the proof we may assume that  $K$ is convex with $B(0,\rho)\subseteq K$ for a fixed $\rho>0$. 
Choose $0<a<\rho/(2\mu(\La))$ and put $r=2a\mu(\La)<\rho$. If $x\in K$, then $B((1-r/\rho)x,r)\subseteq K$ and $B((1-r/\rho)x,r)$ contains 
a point $z\in a\La$. Since 
$$
|x-z|\le r+({r}/{\rho})|x|\le 2a\mu(\La)\left(1+\text{diam}(K)/\rho\right),
$$
we get
$$
K\subseteq (K\cap a\La) +B\big(0,2a\mu(\La)\left(1+\text{diam}(K)/\rho \right)\big),
$$
which completes the argument. 
\end{proof}

Thus Theorems \ref{converge} and \ref{IVconverge} and Corollary \ref{corNew} together with Lemma \ref{dHbounds} yield the following result.
\begin{corollary}\label{convercor}
If $K $ is compact and topologically regular, then
\begin{align*}
&\lim_{a\to 0_+} d_{bL}(\mathcal{V}_{R}^{r,s}(K;\cdot),\mathcal{V}_{R}^{r,s}(K\cap a\La;\cdot)) = 0,\\
&\lim_{a\to 0_+} \mathcal{V}_{R}^{r,s}(K\cap a\La) = \mathcal{V}_{R}^{r,s}(K).
\end{align*}
If, in addition, $K$ has positive reach, then
\begin{align}\label{multigrid}
&\lim_{a\to 0_+} \hat{\Phi}^{r,s}_k(K\cap a\La) = {\Phi}^{r,s}_k(K).
\end{align}
If $K$ is $\delta$-regular or a topologically regular convex set, then the speed of convergence is $O(a)$ when $r=s=0$ and $O(\sqrt{a})$ otherwise. 
\end{corollary} 
The property \eqref{multigrid} means  that   $\hat{\Phi}^{r,s}_k(K\cap a\La)$ is multigrid convergent for the class of sets of positive reach as defined in the introduction. A similar statement about local tensors, but without the speed of convergence, can be made. We omit this here.

\subsection{Possible refinements of the algorithm for digital images}\label{refinement}
We first describe how the number of necessary radii $R_0<R_1<\ldots<R_d$ in \eqref{defEst} can be reduced by one if $s=0$ and $A=\R^d$. Setting $s=0$ and $A=\R^d$ and subtracting
 $(r!)\Phi_d^{r,0}(K)$ on both sides of Equation \eqref{steiner} yields 
\begin{align}\label{modstein}
\int_{K^R\backslash K} p_K(x)^r \,dx = \mathcal{V}_{R}^{r,0}(K)-(r!)\Phi_d^{r,0}(K) = (r!) \sum_{k=1}^d \kappa_{k} R^{k} \Phi_{d-k}^{r,0}(K).
\end{align}
As mentioned in Section \ref{volten}, the volume tensor  $\Phi_d^{r,0}(K)$ can be estimated by 
	$\hat{\phi}_d^{r,0}(K\cap a\La)$.
We may take $\mathcal{V}_{R}^{r,0}(K\cap a\La)-(r!)\hat{\phi}_d^{r,0}(K\cap a\La)$ as an improved estimator for \eqref{modstein}.
This corresponds to replacing the integration domains $B(x,R)\cap V_x(K\cap a\La)$ in \eqref{algorithm2} by 
\[
(B(x,R)\cap V_x(K\cap a\La))\backslash V_x(a\La).
\] 
This makes sense since $V_x(a\La)$ is likely to be contained in $K$ while the left-hand side of  \eqref{modstein} is an integral over $K^R\backslash K$.  The Minkowski tensors can now be isolated from only  $d$ equations of the form \eqref{modstein} with $d$ different values of $R$.

We now suggest a slightly modified estimator for the Minkowski tensors satisfying the same convergence results as $\hat{\Phi}_k^{r,s}(K\cap a\La)$ but where the  number of summands in \eqref{algorithm2} is considerably reduced. As the volume tensors can easily be estimated with the estimators in  Section \ref{volten}, we focus on the tensors with $k<d$.  

Let $K$ be a compact set.  We define the {\em Voronoi neighborhood} $N_\La(0)$ of $0$ to be the set of points $y\in \La$ such that 
the Voronoi cells $V_0(\La)$ and $V_y(\La)$ of $0$ and $y$, respectively, have exactly one common $(d-1)$-dimensional face. 
Similarly, for $z\in \La$ the Voronoi neighborhood $N_\La(z)$ of $z$ is defined, and thus clearly $N_\La(z)=z+N_\La(0)$. 
When $\La\subset \R^2$ is the standard lattice, $N_\La(z)$ consists of the four points in $\La$ that are neighbors of $z$ in the usual $4$-neighborhood \cite{OM}. 
Define
$I(K\cap a\La)$ to be the set of points $z\in K\cap a\La$ such that $N_{a\La}(z)\subseteq K\cap a\La$.  
 The relative complement $B(K\cap a\La)=(K\cap a\La)\setminus I(K\cap a\La)$ of $I(K\cap a\La)$ can be considered as the set of lattice points in $K\cap a\La$ that are close to the boundary of the given set $K$. 

 We modify \eqref{algorithm2} by removing contributions from $I(K\cap a\La)$ and define
 \begin{equation}\label{algorithm3}
 \tilde{\mathcal{V}}_{R}^{r,s}(K\cap a\La ;A)= \sum_{x\in B(K  \cap a\La) \cap A } x^r \int_{B(x,R)\cap V_x(K\cap a\La)}  (y-x)^s\, dy.
 \end{equation}
Assuming that $K$ has positive reach, let $0<R_0<R_1<\ldots<R_d< \textrm{Reach}(K)$. We write again $K_0$ for $K\cap a\La$. Then we obtain the estimators
\begin{align}
\begin{pmatrix}
{\tilde{\Phi}}_{d}^{r,s}(K_0;A\times S^{d-1})\\
\vdots
\\
{\tilde{\Phi}}_{0}^{r,s}(K_0;A\times S^{d-1})
\end{pmatrix}
=\left(A_{R_0,\ldots,R_d}^{r,s}\right)^{-1}
\begin{pmatrix}
\tilde{\mathcal{V}}_{R_0}^{r,s}(K_0;A)\\
\vdots
\\
\tilde{\mathcal{V}}_{R_d}^{r,s}(K_0;A)
\end{pmatrix}\label{defEstcheck}
\end{align}
with $A_{R_0,\ldots,R_d}^{r,s}$ given by \eqref{matrixA}. 

Working with $\tilde{\mathcal{V}}_{R}^{r,s}(K\cap a\La;A)$ reduces the workload  considerably. For instance, when $K$ is $\delta$-regular or polyconvex and topologically regular, the number of elements in $I(K\cap a\La)$ increases with $a^{-d}$, whereas the number of elements in $B(K  \cap a\La)$ only increases with $a^{-(d-1)}$ as $a\to 0_+$. The set $I(K\cap a\La)$ can be obtained from the digital image of $K$ in linear time using a linear filter.
 Moreover, we have the following  convergence result.
 
 \begin{proposition} 
Let $K$ be a topologically regular compact set with positive reach and let $C$ be such that $V_0(\La)\subseteq B(0,C)$. If $A$ is a Borel set in $\R^d$ and $aC<R_0<R_1<\ldots<R_d<\mathrm{Reach}(K)$ and $K_0=K\cap a\La$, then 
\[
\tilde{\Phi}_{k}^{r,s}(K_0;A\times S^{d-1})=\hat{\Phi}_{k}^{r,s}(K_0;A\times S^{d-1})
\]
for all $k\in\{0,\ldots,d-1\}$, whenever $s=0$ or $s$ is odd. 
If $s$ is even and $k\in\{0,\ldots,d-1\}$, then
\begin{equation*}
\lim_{a\to 0_+} \tilde{\Phi}_{k}^{r,s}(K_0;A\times S^{d-1})=\lim_{a\to 0_+}\hat{\Phi}_{k}^{r,s}(K_0;A\times S^{d-1}).
\end{equation*}
 \end{proposition} 

 \begin{proof} Let $aC<R<\mathrm{Reach}(K)$. 
 For $x\in I(K\cap a\La)$, we have  
 \[
 B(x,R)\cap V_{x}(K\cap a\La)=V_{x}(a\La),
 \]
  so the contribution of 
 $x$ to the sum in \eqref{algorithm2} is $(s!)x^r\Phi^{s,0}_d(V_{0}(a\La))$. It follows that 
 \begin{align}\label{Vred}
 {\mathcal{V}}_{R}^{r,s}(K\cap a\La ;A)-\tilde{\mathcal{V}}_{R}^{r,s}(K\cap a\La ;A)=
  (s!)\Phi^{s,0}_d(V_{0}(a\La))\sum_{x\in I(K\cap a\La)\cap A}x^r.
 \end{align}
 For odd $s$ we have $\Phi^{s,0}_d(V_{0}(a\La))=0$, so the claim follows.  For $s=0$ the right-hand side of 
 \eqref{Vred} does not vanish, but it is independent of $R$. A combination of 
 \[
 \left(A_{R_0,\ldots,R_d}^{r,0}\right)^{-1}
 \begin{pmatrix}
1\\1\\
 \vdots
 \\
1 \end{pmatrix}=
 \begin{pmatrix}
(r!)^{-1}\\
0\\
 \vdots
 \\
 0 \end{pmatrix},
 \]
with \eqref{Vred}, \eqref{defEst} and \eqref{defEstcheck} gives the claim. 

For even $s>0$, we have that $\Phi^{s,0}_d(V_{0}(a\La))=a^{d+s}\Phi^{s,0}_d(V_{0}(\La))$, while 
\begin{align*}
\left|\sum_{x\in I(K\cap a\La)\cap A}x^r \right| &\leq \sum_{x\in I(K\cap a\La)}|x|^r  \\
&\leq \sup_{x\in K}|x|^r\sum_{x\in I(K\cap a\La)} 
 \left(a^{d}{\mathcal H}^d(V_{0}(\La))\right)^{-1}{\mathcal H}^d(V_{x}(a\La))\\
&\leq \sup_{x\in K}|x|^r \cdot a^{-d}\cdot {\mathcal H}^d(V_0(\La))^{-1}\cdot \mathcal{H}^d(K^{aC}).
\end{align*}
 Therefore, the expression on the right-hand side of \eqref{Vred} converges to $0$. 
 \end{proof}

It should be noted that a similar modification for $\overline \Phi_k^{r,s}$ is not necessary. In fact the modified Voronoi tensor measure \eqref{modify} with $K=K_0$ has the advantage that small Voronoi cells that are completely contained in the $R_0/2 $-parallel set of $K\cap a\La$ do not contribute. In particular, contributions from $I(K\cap a\La)$ are automatically ignored when $a$ is sufficiently small.

\section{Comparison to known estimators}\label{known}
Most {existing} estimators of intrinsic volumes \cite{digital,lindblad,OM} and Minkowski tensors \cite{turk,mecke} are $n$-local for some $n\in \N$. The idea is to look at all $n\times \dotsm \times n$ pixel blocks in the image and count how many times each of the $2^{n^d}$ possible configurations of black and white points occur. Each configuration is weighted by an element of $\mathbb{T}^{r+s}$ and $\Phi^{r,s}_k(K)$ is estimated as a weighted sum of the configuration counts. It is known that estimators  of this type  for intrinsic volumes other than ordinary volume are not multigrid convergent, even when $K$ is known to be a convex polytope; see \cite{am3}.
{It is not difficult to see that there cannot be a multigrid convergent $n$-local estimator for the (even rank) tensors $\Phi_k^{0,2s}(K)$ with $k=0,\ldots,d-1$, $s\in\mathbb{N}$, for polytopes $K$, either. In fact, repeatedly taking the trace of such an estimator would lead to a multigrid convergent $n$-local estimator of the $k$th intrinsic volume, in contradiction to \cite{am3}.}

The algorithm presented in this paper is not $n$-local for any $n\in \N$. It is  required in the convergence proof that the parallel radius $R$ is fixed while the resolution $a^{-1}$ goes to infinity. {The non-local operation in the definition of our estimator is the calculation of the Voronoi diagram.} 
The computation time for Voronoi diagrams of $k$ points is $O(k\log k + k^{\lfloor d/2\rfloor})$, see \cite{chazelle}, which is somewhat slower than  $n$-local algorithms for which the computation time for $k$ data points is $O(k)$. The computation time can be improved by ignoring interior points as discussed in Section \ref{refinement}.

The idea to base digital estimators for intrinsic volumes on an inversion of the Steiner formula as in \eqref{matrix} has occurred before in \cite{spodarev,jan}. In both references, the authors  define  estimators for polyconvex sets which  are not necessarily of positive reach. This more ambitious aim leads to problems with the convergence.

In \cite{spodarev}, the authors use  a version of the Steiner formula for polyconvex sets given in terms of the Schneider index, see \cite{schneider}. Since its definition is, however, $n$-local in nature, the authors choose an $n$-local algorithm to estimate it. As already mentioned, such algorithms are not multigrid convergent. 

In \cite{jan}, it is used that the intrinsic volumes of a polyconvex set can, on the one hand, be approximated by those of a parallel set with small parallel radius, and on the other hand, the closed complement of this parallel set has positive reach, so that its intrinsic volumes can be  computed via the Steiner formula. The authors employ a discretization of the parallel volumes of digital images, but without showing that the convergence is preserved. 

It is likely that the ideas of the present paper combined with the ones of \cite{jan} could be used to construct {multigrid} convergent digital algorithms for polyconvex sets. The price for this is that the notion of convergence in \cite{jan} is slightly artificial for practical purposes, requiring very small parallel radii in order to get good approximations and at the same time large radii compared to resolution.

In \cite{svane}, $n$-local  algorithms based on grey-valued images are suggested. They are shown to converge to the true value when the resolution {tends} to infinity.  However, they only apply to surface and certain mean curvature tensors. Moreover, they are hard to apply in practice, since they require detailed information about the underlying point spread function {which specifies the representation of the object as grey-value image. If grey-value images are given, the}
 algorithm of the present paper could be applied to thresholded images, but there may be more efficient ways to exploit the additional information of the grey-values.
\bigskip

\section*{acknowledgements}
We wish to thank the referees for carefully reading the paper and making helpful suggestions for improvements. The first author was supported in part by DFG grants FOR 1548 and HU 1874/4-2. The third author was supported by a grant from the Carlsberg Foundation.
The second and third authors were supported by the Centre for Stochastic
Geometry and Advanced Bioimaging, funded by the Villum Foundation.

\end{document}